\begin{filecontents*}{example.eps}
\documentclass[11pt]{article}
\usepackage{latexsym,amsmath,amsthm,verbatim,ifthen,amssymb}
\usepackage[all]{xy}

\usepackage[utf8]{inputenc}

 \usepackage{color}

\newtheorem{theorem}{Theorem}[section]

 \newtheorem{proposition}[theorem]{Proposition}
 \theoremstyle{definition}
 \newtheorem{definition}[theorem]{Definition}
 \theoremstyle{remark}

 \numberwithin{equation}{subsection}

\usepackage[margin=3cm]{geometry}
\usepackage{etoolbox,xparse}
\usepackage[T1]{fontenc}
\usepackage[utf8]{inputenc}
\usepackage{lmodern}
\usepackage[english]{babel}
\usepackage{tikz-cd}
\usepackage{mathtools}
\usepackage{amssymb}
\usepackage{graphicx}
\usepackage[shortlabels,inline]{enumitem}
\usepackage{booktabs}
\usepackage{physics}
\usepackage{multicol} 
\usepackage{pgfplots}
\usepackage{cancel} 
\usepackage{dsfont}
\usepackage{bm} 
\usepackage{amsthm,thmtools} 

\usepackage{float}
\usepackage{amsrefs}

\newcommand{\C}{\mathbb{C}}
\newcommand{\E}{\underline{E}}
\newcommand{\F}{\mathcal{F}}
\newcommand{\HH}{\mathbb{H}}
\newcommand{\K}{\mathbb{K}}

\newcommand{\QQ}{\mathcal{Q}}
\newcommand{\R}{\mathbb{R}}
\renewcommand{\S}{{S}}  

\newcommand{\Z}{\mathbb{Z}}

\newcommand{\Hom}{\text{Hom}}
\newcommand{\pt}{\text{pt}}
\newcommand{\catname}[1]{{\normalfont\textbf{#1}}}
\newcommand{\Or}{{\catname{Or}_\F(G)}}

\DeclarePairedDelimiter\floor{\lfloor}{\rfloor}
\DeclarePairedDelimiter\ceil{\lceil}{\rceil}

\usepackage[shortlabels,inline]{enumitem}
\usepackage{tikz-cd}
\usepackage{hyperref}
\hypersetup{     
    colorlinks,
    citecolor=blue,
    filecolor=blue,
    linkcolor=blue,
    urlcolor=blue
}

\makeatletter
\def\blfootnote{\xdef\@thefnmark{}\@footnotetext}
\makeatother

\begin{document}

\title{The equivariant $K$- and $KO$-theory of certain classifying spaces via an equivariant Atiyah-Hirzebruch spectral sequence}


\author{Mario Fuentes}



\maketitle

\begin{abstract}
We compute the $K$- and $KO$-theory for the classifying $G$-spaces $\underline EG$ for proper actions of certain infinite discrete groups $G$ via the equivariant Atiyah-Hirzebruch spectral sequence.
\end{abstract}

\section*{Introduction}\label{intro}

\blfootnote{2010 {\em Mathematics Subject Classification}. Primary 19L47, 55N91. Secondary: 55R91, 20F65, 20F55.}
\blfootnote{{\em Key words}. Classifying spaces for proper actions, equivariant $K$-theory, equivariant $KO$-theory, amalgamated product of finite cyclic groups, Coxeter groups.}

The classifying space $\underline EG$ of proper actions of a given discrete group 
has recently become increasingly important. It is a $G$-CW complex whose isotropy groups are finite subgroups of $G$ and such that the fixed point set of finite subgroups of $G$ are weakly contractible spaces.  The equivariant $K$- and $KO$-theory of these classifying spaces were first studied by Lück and Oliver \cite{WL} and they are closely related to the Baum-Connes Conjecture (see \cite{BCH94} for the original reference or \cite{Apa} for an extended survey about this topic). In this paper we compute $K_G^*(\E G)$ and $KO^*_G(\E G)$  whenever $G$ is an amalgamated product of finite cyclic groups or certain Coxeter groups, see Theorems \ref{amalgamado}, \ref{cox1}, \ref{non1} and \ref{non2} for the precise statements. 

To this end, we first describe in Section 1 the equivariant Atiyah-Hirzebruch spectral sequence which relates the Bredon cohomology introduced in \cite{WL} with the equivariant $K-$ and $KO$-theory. 
The Bredon cohomology of various spaces is attracting the attention of researchers (see for example 
\cite{Bar}, for a complete survey).
For each one of the studied groups, we present suitable models for their classifying spaces of proper actions which let us compute the Bredon cohomology and the equivariant $K$- and $KO$-theory.

The results in this paper are by no means exhaustive and just illustrate how the used method may serve as an algorithm for similar computations. However, the examples studied are of interest: the family of amalgamated products of finite cyclic groups contains some classical groups as the infinite dihedral group  $D_\infty\cong \Z_2\ast \Z_2$, the special linear group $SL_2(\Z)\cong \Z_6\ast_{\Z_2}\Z_4$ and the projective special linear group $PSL_2(\Z)\cong \Z_3\ast\Z_2$. In the case of Coxeter groups \cite{Coxetergroups}, their relation with the Baum-Connes Conjecture of such groups have been studied, for example, in \cite{CoxeterRef1} or \cite{CoxeterRef2}. Although a general method for computing $K^*_G(\E G)$ and $KO^*_G(\E G)$ for $G$ a Coxeter group is provided, only a few infinite Coxeter groups present such a simple model that we are able to perform calculations in practice.

\begin{center} ACKNOWLEDGMENTS \end{center}
I would like to express my sincere gratitude to Irakli Patchkoria and Wolfgang Lück for their guidance.

\section{Preliminaries and the equivariant Atiyah-Hirzebruch spectral sequence}

Throughout this text $G$ is a discrete group. Here, we briefly recall the main facts about equivariant $K$-theory for finite proper $G$-CW-complexes developed in \cite{WL}. 

A $G$-space $X$ is a {\em $G$-CW-complex} if there is a filtration $\emptyset=X_{-1}\subset X_{0}\subset \ldots \subset X_n\subset\dots$ such that $X$ is the weak union of the spaces $X_n$ and each $X_{n}$ is obtained from $X_{n-1}$ by attaching some orbits of cells. That is, there are sets $\{H_j \le G\}_{j\in J_n}$ of subgroups of $G$ and attaching maps $q_j:G/H_j\times \S^{n-1}\to X_{n-1}$ such that

$$\begin{tikzcd}
\bigsqcup_{j\in J_n} G/H_j\times \S^{n-1} \arrow{r}{\bigsqcup_{j\in J_n} q_j} \arrow{d} & X_{n-1}\arrow{d}\\
\bigsqcup_{j\in J_n} G/H_j\times D^n \arrow{r}{\bigsqcup_{j\in J_n} Q_j} & X_{n}\\
\end{tikzcd}$$
is a pushout square. If $A\subset X$ is a $G$-invariant subcomplex, then $(X, A)$ is a {\em $G$-CW-pair}. A $G$-CW-complex is {\em finite} if it has finitely many orbits of cells and it is {\em proper} if its isotropy groups are finite.

A {\em $G$-vector bundle} over a a $G$-CW-complex $X$ is a complex vector bundle $p\colon E\to X$  where the total space $E$ is a $G$-space, $p$ is $G$-equivariant and each $g\in G$ acts as a bundle isomorphism on $E$.  If the space $X$ is a point, a $G$-vector bundle over $X$ is just a linear representation of the group $G$.

The {\em equivariant $K$-theory} of $X$ is defined as follows: $K_G(X)=K^0_G(X)$ is the Grothendieck group of the monoid of isomorphism classes of $G$-vector bundles over $X$. If $n>0$,
  $$K^{-n}_G(X)=\ker(K^0_G(X\times \S^n) \xrightarrow{\text{incl}^*} K^0_G(X))$$
 where the $G$-action on $\S^n$ is imposed to be trivial. For $(X,A)$ a proper $G$-CW-pair,
  $$K_G^{-n}(X,A)=\ker(K_G^{-n}(X\cup_AX)\xrightarrow{i_2^*}K_G^{-n}(X)).$$
 In positive degree, the equivariant $K$-theory is defined using  Bott periodicity. If $G=\{e\}$, then $K^*_G=K^*$ is the ordinary non-equivariant $K$-theory.

 \cite[Theorem 3.2]{WL} asserts that  $K_G^*$ is an equivariant cohomology theory: it is a $G$-homotopy invariant contravariant functor, from the category of finite proper $G$-CW-pairs to the category of graded abelian groups satisfying Mayer-Vietoris, excision, exactness. Moreover, it also satisfies Bott periodicity of order 2, i.e. $K_G^{-n}(X,A)\cong K_G^{-n-2}(X,A)$.

For $H<G$ a finite subgroup, we have a natural isomorphism $K^0_G(G/H)\cong R(H)$, where $R(H)$ is the (complex) representation ring of the group $H$ and $K^1_G(G/H)=0$. 

All of the above can be redone using \textit{real} (instead of complex) $G$-vector bundles to obtain $KO^*_G$, the real equivariant $K$-theory. In the real case, Bott periodicity is of order 8, $KO^n_G(X)\cong KO_G^{n-8}(X)$, and for any finite subgroup $H<G$, $KO^*_G(G/H)\cong KO^*_H(\pt)$. In particular, $KO^0_G(G/H)\cong RO(H)$, where $RO(H)=R_\R(H)$ is the representation ring over $\R$.

We now recall the  {\em Bredon cohomology} introduced in  \cite[Sec.\ 2]{WL}. Let $\mathcal F$ be the a family of subgroups of $G$ (in our examples $\mathcal F$ will be the family of finite subgroups of $G$) and let $M$ be a contravariant functor from $\Or$, the orbit category associated to $\mathcal F$, to $\textbf{Ab}$, the category of abelian groups.

For a given $G$-CW-complex, denote $C_n(X): \Or^{\operatorname{op}} \to \textbf{Ab}$ the contravariant functor which sends an orbit $G/H$ to the free abelian group with one generator for each $n$-cell in $X^H$, i.e. the fixed points under the action of $H$. A morphism $G/H\to G/K$ is sent by $C_n(X)$ to the induced inclusion or permutation of cells (note that a map in the orbit category is generated by inclusion of subgroups and conjugations).

The {\em Bredon cohomology} of $X$ with coefficients in $M$ is defined as the homology of the cochain complex
$$
0\rightarrow \Hom_\Or({C_0}(X),M)\xrightarrow{\delta}
\Hom_\Or({C_1}(X),M) \xrightarrow{\delta}
\Hom_\Or({C_2}(X),M)\rightarrow \cdots
$$
where the morphism $\delta$ is induced by the boundary map of the cellular structure.

Finally, 
 we present a  version of the equivariant {\em Atiyah-Hirzebruch spectral sequence} whose $E_2$-term is the {\em Bredon cohomology}  recalled above. This result can be obtained as  a particular instance of \cite[Thm.\ 4.7(2)]{Assembly}.

\begin{theorem}\label{Spectral Sequence Theorem}
Given $X$ a finite proper $G$-CW-complex, and $h^*$ a $G$-equivariant cohomology theory, there is a  cohomological spectral sequence whose $E_2$-page is $H_G^p(X; h^q)$ which converges to $h^{p+q}$:
$$
E^{p,q}_2=H^p_G(X;h^q)\implies h^{p+q}(X).
$$
\end{theorem}

\section{$K$- and $KO$-theory of  classifying spaces for proper actions}

We begin with the following.

\begin{definition}\label{def1}\cite[Def.\ 1.8]{Survey} Let
 $\F$ be a family of subgroups of $G$ which is closed under conjugation and finite intersections.
 The {\em classifying $G$-CW-complex for $\F$}  is   a $G$-CW-complex $E_\F(G)$ whose  isotropy groups belong to $\F$ and  it is universal with that property. That is, for any $G$-CW-complex $Y$ whose  isotropy groups belong to $\F$ there is only one $G$-map $Y\to E_\F(G)$ up to $G$-homotopy.

For a given family $\F$, its classifying space exists and it is characterized by the fact that all its isotropy subgroups lie in $\F$ and, for each $H\in \F$, the $H$-fixed point set is weakly contractible \cite[Thm.\ 1.9]{Survey}.

When $\F$ is the family of finite subgroups of $G$ denote $E_\F(G)$ by $\underline{E}G$  and call it the {\em classifying $G$-CW-complex for proper $G$-actions}.

\end{definition}

 When $G$ is torsion free, $\underline{E}G$ coincides with $EG$ 
 but in general they differ.
Using the Atiyah-Hirzebruch spectral sequence as the main tool, we plan to compute the equivariant $K$- and $KO$-theory of $\underline{E}G$ for some discrete groups $G$. The relevance of this lies, for instance, in the Baum-Connes Conjecture \cite{BCH94} which links a certain dual of $K_G^*(\underline{E}G)$ with the $K$-theory of the reduced $C^*$-algebra of $G$, see also  \cite{Bestvina} and \cite[Sec.\ 7.1]{Survey}.

\subsection{Amalgamated products of finite cyclic groups}

Here we prove:

\begin{theorem}\label{amalgamado} Let $G$ be the amalgamated product $\Z_{r_1m_0} \ast_{\Z_{r_1}} \Z_{r_1r_2m_1} \ast_{\Z_{r_2}}\cdots \ast_{\Z_{r_k}}\Z_{r_km_k}$ for some positive integer numbers $r_1,\dots r_k, m_0,m_1,\dots,m_k$. Then:
$$K^{-n}_{G}(\E G)=\left\{ \begin{array}{ll}
\Z^{\sigma}, & \text{ if } n \text{ is even,}\\
0, & \text{ if } n \text{ is odd,}\\
\end{array}\right.$$ where $\sigma=\sum_{i=0}^k m_ir_{i-1}r_i - \sum_{i=1}^k r_i$ with $r_0=1=r_{k+1}$.

In the real case, if $r_i$ are all odd,
$$KO^{-n}_{G}(\E G)=\left\{ \begin{array}{cl}
\Z^\sigma, & \text{ if } n\equiv 0,4 \text{ (mod 8),}\\
(\Z_2)^\omega, & \text{ if } n\equiv 1 \text{ (mod 8),}\\
(\Z_2)^\omega\oplus \Z^\theta, & \text{ if } n\equiv 2 \text{ (mod 8),}\\
\Z^\theta, & \text{ if } n\equiv 6 \text{ (mod 8,)}\\
0, & \text{ if } n\equiv 3,5,7 \text{ (mod 8),}\\
\end{array}\right.$$
where,

$$ \sigma=\sum_{i=0}^k \floor{\frac{m_ir_ir_{i+1}}{2}} -\frac{1}{2}\sum_{i=1}^k r_i +\frac{1}{2}k+1,$$

$$\omega= \frac{1}{2} \sum_{i=0}^k (-1)^{m_i} +\frac{k}{2}+\frac{3}{2},$$

$$ \theta=\sum_{i=0}^k \ceil{\frac{m_ir_ir_{i+1}}{2}} -\frac{1}{2}\sum_{i=1}^k r_i -\frac{k}{2}-1.$$

\end{theorem}

Before proving the theorem, we use that following proposition which gives a suitable model for $\E G$.
We say that a group acts on a graph without inversions if no edge is reversed by any element of the group. The following proposition is a direct consequence of \cite[Theorem 4.7]{Survey}

\begin{proposition}\label{Theorem Tree Model}
Suppose that $G$ acts on a tree $X$ without inversions and the isotropy group of any vertex is finite. Then $X$ is a model for $\E G$.
\end{proposition}

\begin{proof}[Proof of Theorem \ref{amalgamado}]
We first find a tree in which the group $G=\Z_{r_1m_0} \ast_{\Z_{r_1}} \Z_{r_1r_2m_1} \ast_{\Z_{r_2}}\cdots
\ast_{\Z_{r_k}}\Z_{r_km_k}$  acts without inversion and with finite isotropy groups at vertices.
For it, we follow the algorithm given in the  fundamental theorem of the Bass-Serre theory (see \cite[Theorem 9]{Serre}). We include it here for completeness:

Start by a vertex $p_0$ which is going to be of `type 0' and add $m_0$ edges adjacent to it. At the end of these edges add vertices of `type 1'. For each one of these type 1 vertices construct $2m_1$ edges adjacent to it (including the one which joins it with $p_0$). Alternatively put vertices of `type 0' and `type 2' at the end of these edges. Successively add edges and vertices depending on the type. Fix a path $p_0, p_1, \dots p_k$ of vertices of type $0,1,\dots k$ respectively.
At the end we will get a tree $X$ such that the action of the generators of $G$ is given by a rotation around $p_i$. The quotient graph is a path which connects  $p_0,p_1,\ldots, p_k$. See the figure \ref{Figure tree} for an example of this construction.

\begin{figure}[h!]
\centering

\begin{tikzpicture}[scale=0.8]

\draw[-] (-2.6-0.6*2,0)--(2.6+0.6*4+0.6*2,0);
\draw[-] (0,2.6+0.6*2)--(0,-2.6-0.6*2);

\draw[-] (2.6,0)--(2.6-0.6*1,0.6*1.01);
\draw[-] (2.6,0)--(2.6+0.6*1,0.6*1.01);
\draw[-] (2.6,0)--(2.6-0.6*1,-0.6*1.01);
\draw[-] (2.6,0)--(2.6+0.6*1,-0.6*1.01);
\draw[-] (-2.6,0)--(-2.6-0.6*1,0.6*1.01);
\draw[-] (-2.6,0)--(-2.6+0.6*1,0.6*1.01);
\draw[-] (-2.6,0)--(-2.6-0.6*1,-0.6*1.01);
\draw[-] (-2.6,0)--(-2.6+0.6*1,-0.6*1.01);

\draw[-] (0,2.6)--( 0.6*1.01 , 2.6+0.6);
\draw[-] (0,2.6)--( -0.6*1.01 , 2.6+0.6);
\draw[-] (0,2.6)--( 0.6*1.01 , 2.6-0.6);
\draw[-] (0,2.6)--( -0.6*1.01 , 2.6-0.6);
\draw[-] (0,-2.6)--( 0.6*1.01 , -2.6+0.6);
\draw[-] (0,-2.6)--( -0.6*1.01 , -2.6+0.6);
\draw[-] (0,-2.6)--( 0.6*1.01 , -2.6-0.6);
\draw[-] (0,-2.6)--( -0.6*1.01 , -2.6-0.6);

\draw[thick, fill=white] (2.6,0) circle [radius=0.1];
\draw[thick, red, fill=red] (0,0) circle [radius=0.1];
\draw[thick, fill=white] (-2.6,0) circle [radius=0.1];
\draw[thick, fill=white] (0,2.6) circle [radius=0.1];
\draw[thick, fill=white] (0,-2.6) circle [radius=0.1];

\draw[thick, fill=black] (2.6-0.6*1,0.6*1.01)  circle [radius=0.1];
\draw[thick, fill=black] (2.6-0.6*1,-0.6*1.01) circle [radius=0.1];
\draw[red, thick, fill=red] (2.6+0.6*1,0.6*1.01)  circle [radius=0.1];
\draw[red, thick, fill=red] (2.6+0.6*1,-0.6*1.01) circle [radius=0.1];
\draw[thick, fill=black] (2.6+0.9*2,0) circle [radius=0.1];

\draw[thick, fill=black] (-2.6+0.6*1,0.6*1.01)  circle [radius=0.1];
\draw[thick, fill=black] (-2.6+0.6*1,-0.6*1.01) circle [radius=0.1];
\draw[red, thick, fill=red] (-2.6-0.6*1,0.6*1.01)  circle [radius=0.1];
\draw[red, thick, fill=red] (-2.6-0.6*1,-0.6*1.01) circle [radius=0.1];
\draw[thick, fill=black] (-2.6-0.6*2,0) circle [radius=0.1];

\draw[thick, fill=black] ( 0.6*1.01 , 2.6-0.6)  circle [radius=0.1];
\draw[thick, fill=black] ( -0.6*1.01 , 2.6-0.6)  circle [radius=0.1];
\draw[red, thick, fill=red] ( 0.6*1.01 , 2.6+0.6)  circle [radius=0.1];
\draw[red, thick, fill=red] ( -0.6*1.01 , 2.6+0.6)  circle [radius=0.1];
\draw[thick, fill=black] (0,2.6+0.6*2) circle [radius=0.1];

\draw[thick, fill=black] ( 0.6*1.01 , -2.6+0.6)  circle [radius=0.1];
\draw[thick, fill=black] ( -0.6*1.01 , -2.6+0.6)  circle [radius=0.1];
\draw[red, thick, fill=red] ( 0.6*1.01 , -2.6-0.6)  circle [radius=0.1];
\draw[red, thick, fill=red] ( -0.6*1.01 , -2.6-0.6)  circle [radius=0.1];
\draw[thick, fill=black] (0,-2.6-0.6*2) circle [radius=0.1];


\draw[-] (2*2.6+0.9*4+0.6*0-2.6-0.6*2,0)--(2*2.6+0.9*4+0.6*0+2.6+0.6*2,0);
\draw[-] (2*2.6+0.9*4+0.6*0,2.6+0.6*2)--(2*2.6+0.9*4+0.6*0,-2.6-0.6*2);

\draw[-] (2*2.6+0.9*4+0.6*0+2.6,0)--(2*2.6+0.9*4+0.6*0+2.6-0.6*1,0.6*1.01);
\draw[-] (2*2.6+0.9*4+0.6*0+2.6,0)--(2*2.6+0.9*4+0.6*0+2.6+0.6*1,0.6*1.01);
\draw[-] (2*2.6+0.9*4+0.6*0+2.6,0)--(2*2.6+0.9*4+0.6*0+2.6-0.6*1,-0.6*1.01);
\draw[-] (2*2.6+0.9*4+0.6*0+2.6,0)--(2*2.6+0.9*4+0.6*0+2.6+0.6*1,-0.6*1.01);
\draw[-] (2*2.6+0.9*4+0.6*0-2.6,0)--(2*2.6+0.9*4+0.6*0-2.6-0.6*1,0.6*1.01);
\draw[-] (2*2.6+0.9*4+0.6*0-2.6,0)--(2*2.6+0.9*4+0.6*0-2.6+0.6*1,0.6*1.01);
\draw[-] (2*2.6+0.9*4+0.6*0-2.6,0)--(2*2.6+0.9*4+0.6*0-2.6-0.6*1,-0.6*1.01);
\draw[-] (2*2.6+0.9*4+0.6*0-2.6,0)--(2*2.6+0.9*4+0.6*0-2.6+0.6*1,-0.6*1.01);

\draw[-] (2*2.6+0.9*4+0.6*0+0,2.6)--(2*2.6+0.9*4+0.6*0+ 0.6*1.01 , 2.6+0.6);
\draw[-] (2*2.6+0.9*4+0.6*0+0,2.6)--(2*2.6+0.9*4+0.6*0 -0.6*1.01 , 2.6+0.6);
\draw[-] (2*2.6+0.9*4+0.6*0+0,2.6)--(2*2.6+0.9*4+0.6*0+ 0.6*1.01 , 2.6-0.6);
\draw[-] (2*2.6+0.9*4+0.6*0,2.6)--(2*2.6+0.9*4+0.6*0 -0.6*1.01 , 2.6-0.6);
\draw[-] (2*2.6+0.9*4+0.6*0,-2.6)--(2*2.6+0.9*4+0.6*0+ 0.6*1.01 , -2.6+0.6);
\draw[-] (2*2.6+0.9*4+0.6*0,-2.6)--( 2*2.6+0.9*4+0.6*0-0.6*1.01 , -2.6+0.6);
\draw[-] (2*2.6+0.9*4+0.6*0,-2.6)--(2*2.6+0.9*4+0.6*0+ 0.6*1.01 , -2.6-0.6);
\draw[-] (2*2.6+0.9*4+0.6*0,-2.6)--(2*2.6+0.9*4+0.6*0 -0.6*1.01 , -2.6-0.6);

\draw[thick, fill=white] (2*2.6+0.9*4+0.6*0+2.6,0) circle [radius=0.1];
\draw[thick, red, fill=red] (2*2.6+0.9*4+0.6*0,0) circle [radius=0.1];
\draw[thick, fill=white] (2*2.6+0.9*4+0.6*0-2.6,0) circle [radius=0.1];
\draw[thick, fill=white] (2*2.6+0.9*4+0.6*0+0,2.6) circle [radius=0.1];
\draw[thick, fill=white] (2*2.6+0.9*4+0.6*0+0,-2.6) circle [radius=0.1];

\draw[thick, fill=black] (2*2.6+0.9*4+0.6*0+2.6-0.6*1,0.6*1.01)  circle [radius=0.1];
\draw[thick, fill=black] (2*2.6+0.9*4+0.6*0+2.6-0.6*1,-0.6*1.01) circle [radius=0.1];
\draw[red, thick, fill=red] (2*2.6+0.9*4+0.6*0+2.6+0.6*1,0.6*1.01)  circle [radius=0.1];
\draw[red, thick, fill=red] (2*2.6+0.9*4+0.6*0+2.6+0.6*1,-0.6*1.01) circle [radius=0.1];
\draw[thick, fill=black] (2*2.6+0.9*4+0.6*0+2.6+0.6*2,0) circle [radius=0.1];

\draw[thick, fill=black] (2*2.6+0.9*4+0.6*0-2.6+0.6*1,0.6*1.01)  circle [radius=0.1];
\draw[thick, fill=black] (2*2.6+0.9*4+0.6*0-2.6+0.6*1,-0.6*1.01) circle [radius=0.1];
\draw[red, thick, fill=red] (2*2.6+0.9*4+0.6*0-2.6-0.6*1,0.6*1.01)  circle [radius=0.1];
\draw[red, thick, fill=red] (2*2.6+0.9*4+0.6*0-2.6-0.6*1,-0.6*1.01) circle [radius=0.1];

\draw[thick, fill=black] (2*2.6+0.9*4+0.6*0+ 0.6*1.01 , 2.6-0.6)  circle [radius=0.1];
\draw[thick, fill=black] (2*2.6+0.9*4+0.6*0 -0.6*1.01 , 2.6-0.6)  circle [radius=0.1];
\draw[red, thick, fill=red] (2*2.6+0.9*4+0.6*0+ 0.6*1.01 , 2.6+0.6)  circle [radius=0.1];
\draw[red, thick, fill=red] (2*2.6+0.9*4+0.6*0 -0.6*1.01 , 2.6+0.6)  circle [radius=0.1];
\draw[thick, fill=black] (2*2.6+0.9*4+0.6*0+0,2.6+0.6*2) circle [radius=0.1];

\draw[thick, fill=black] (2*2.6+0.9*4+0.6*0+ 0.6*1.01 , -2.6+0.6)  circle [radius=0.1];
\draw[thick, fill=black] (2*2.6+0.9*4+0.6*0 -0.6*1.01 , -2.6+0.6)  circle [radius=0.1];
\draw[red, thick, fill=red] (2*2.6+0.9*4+0.6*0+ 0.6*1.01 , -2.6-0.6)  circle [radius=0.1];
\draw[red, thick, fill=red] (2*2.6+0.9*4+0.6*0 -0.6*1.01 , -2.6-0.6)  circle [radius=0.1];
\draw[thick, fill=black] (2*2.6+0.9*4+0.6*0,-2.6-0.6*2) circle [radius=0.1];

\draw [->, very thick,blue] (1,0) arc [radius=1, start angle=0, end angle= 90];
\draw [->, very thick,blue] (2.6+0.5,0) arc [radius=0.5, start angle=0, end angle=125];
\draw [->, very thick,blue] (2.6-0.5,0) arc [radius=0.5, start angle=180, end angle=180+125];
\draw [->, very thick,blue] (2.6+0.9*2+0.5,0) arc [radius=0.5, start angle=0, end angle=180];

\node at (1,1) {$\alpha$};
\node at (2.6,1) {$\beta$};
\node at (2.6+2*0.9,1) {$\gamma$};

\end{tikzpicture}
\caption{Example of a tree where $G=\Z_{4}\ast\Z_3\ast\Z_2=\langle \alpha,\beta, \gamma\mid \alpha^4=\beta^3=\gamma^2=e\rangle$ acts without inversions. The vertices of type 0, 1 or 2 are represented by the colors red, white and black respectively.}\label{Figure tree}
\end{figure}
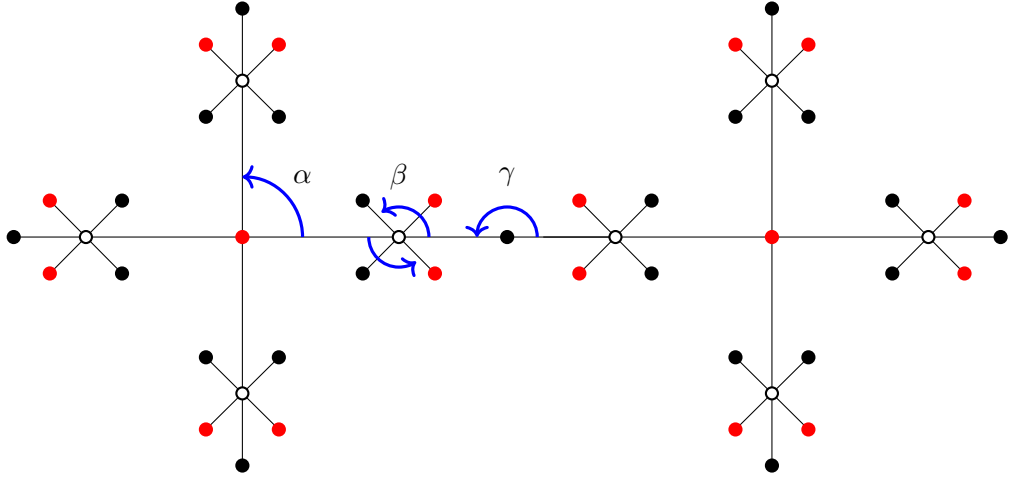

 By Proposition \ref{Theorem Tree Model}, such tree $X$ is precisely  a model for $\underline EG$. Observe that the isotropy group of any vertex is isomorphic to one of the cyclic groups $\Z_{m_i}$, depending on the vertex type.

On the one hand,  since $K^1_G$ is zero in the orbits, $H^*_G(X;K^1_G)=0$. On the other hand, $K^0_G(G/H)\cong R(H)$ with $R(H)$ the representation ring  of the finite subgroup $H$. The multiplicative structure is not being taken into account, so consider $R(H)$ as a free abelian group. In the case of cyclic groups (see \cite[section 5.1]{LinearRepresentation}),
$$R(\Z_m)=\Z^{m}$$
and, if $i\colon \Z_r\hookrightarrow \Z_{mr}$ is the inclusion,
$$R(i)=\varphi_{m,r}\colon\Z^{mr}\to \Z^r,\;\; (a_1,\ldots, a_{mr})\mapsto
\qty( \sum_{i=0}^{m-1} a_{ir+1}, \sum_{i=0}^{m-1} a_{ir+2},\ldots , \sum_{i=0}^{m-1} a_{ir+r}).$$
Then, the Bredon cohomology  is given by the following cochain complex,
$$0\to\Z^{m_0r_1}\oplus\Z^{m_1r_0r_1}\dots \oplus\Z^{m_kr_k}\xrightarrow{\delta} \Z^{r_1}\oplus\dots \oplus \Z^{r_k}\to 0$$
where $(\bar{a}_1,\dots, \bar{a}_k)\in \Z^{m_0r_1}\oplus\dots \oplus\Z^{m_kr_k}$ is mapped by $\delta$ to
$$\textstyle \qty(\varphi_ {m_0,r_1}(\bar{a}_1)-\varphi_{m_1r_2,r_1}(\bar{a}_2), \varphi_{m_1r_1,r_2}(\bar{a}_2)-\varphi_{m_2r_3,r_2}(\bar{a}_3),..., \varphi_{m_{k-1}r_{k-1},r_k}(\bar{a}_{k-1})-\varphi_{m_k,r_k}(\bar{a}_k))$$
in $\Z^{r_1}\oplus\dots\oplus\Z^{r_k}$.

Note that all maps $\varphi_{\bullet,\bullet}$ are surjective. Therefore, $\delta$ is surjective onto the first term of the image, $\Z^{r_1}$. Successively, any element of any term can be proved to be targeted by $\delta$, which implies that $\delta$ is itself surjective. Therefore,  $H^0_G(X;K^0_G)=\Z^\sigma$  where
$\sigma=\sum_{i=0}^k m_ir_{i-1}r_i - \sum_{i=1}^k r_i$
with $r_0=1=r_{k+1}$.

Hence, the Atiyah-Hirzebruch spectral sequence of Theorem \ref{Spectral Sequence Theorem} trivially collapses at the second page $E_2^{p,q}=H^p_G(X;K^q_G)$  and the first statement is proved.

In a similar way, we compute the equivariant $KO$-theory. Recall from \cite[Remark in p. 133]{Segal} that for a trivial   $G$-space $X$,
\begin{equation}\label{ecu2}
KO^0_G(X)\cong (KO^0(X)\otimes R(G;\R))\oplus (K^0(X)\otimes R(G;\C))\oplus (K\HH^0 (X)\otimes R(G;\HH))
\end{equation}
 where $K\HH$ is the quaternionic $K$-theory (see \cite[Sec. 9]{qqq}). Note that the groups $R(G;\K)$ can be obtained as follows:  $RO(G)$ is the free abelian group generated by the simple real $G$-modules. Given such a simple representation $M$, compute its endomorphism field, $D_M=\text{End}(M)$, then we have $D_M=\R,\C$ or $\HH$.
 Finally, $R(G;\K)$ is the free group generated by simple real $G$-modules with endomorphism field $D_M=\K$.

By definition, $KO_{\Z_s}^{-n}(\pt)=\ker(KO^0_{\Z_s}(\S^n)\to KO^0_{\Z_s}(\pt))$. The $KO$-theory of an sphere is well-known. On the other hand, $\Z_s$ has $s$ irreducible real representations. The trivial representation has $\R$ as endomorphism field; also the sign representation does if $s$ is even. Any other representation has $\C$ as endomorphism field.
These facts about the real representation theory of $\Z_s$ can be deduced from the classification given in \cite[p. 108]{LinearRepresentation} and the table of characters of the irreducible representations \cite[section 5.1]{LinearRepresentation}.

Using all these facts, the computation of $KO_{\Z_s}^{-n}(\pt)$ is immediate and it is expressed in the following table:

\begin{table}[ht]
\centering
\begin{tabular}{c |  c  c  c  c  c  c  c  c}
$n$ & 0 &1 &2&$\;3\;$&$\;4\;$&$\;5\;$&$\;6\;$&$\;7\;$\\
[1ex]
\hline
\\
$KO^{-n}_{\Z_s}(\pt)$  & $\Z^{\floor{\frac{s}{2}}+1}$ & $(\Z_2)^{\frac{3+(-1)^s}{2}}$ &  $(\Z_2)^{\frac{3+(-1)^s}{2}}\oplus \Z^{\ceil{\frac{s}{2}}-1}$ & $0$ & $\Z^{\floor{\frac{s}{2}}+1}$& $0$ & $\Z^{\ceil{\frac{s}{2}}-1}$ & $0$ \\

\end{tabular}
\end{table}

Furthermore, it can be checked that all representations in the image of $i^*$,  induced by the inclusions $i\colon \Z_r\hookrightarrow \Z_{mr}$,  are targeted, with the possible exception of the sign representation. If we impose $r$ to be odd, there is no such sign representation in $KO^{-n}_{\Z_s}(\pt)$, and the induced map $i^*$ is always surjective. This implies that, when the Bredon cohomology is computed, the degree 1 cohomology is annihilated, and the degree 0 cohomology is just the the kernel of the boundary map $\delta$.

In this situation, the $E_2$-page of the spectral sequence has only one non-trivial column, so it collapses giving the stated result.
\end{proof}

\subsection{Right-angled Coxeter Groups}

A group $W$ is called a $\textit{Coxeter group}$ \cite{Coxetergroups} if it admits a presentation of the form,
$$W=\langle S\mid (s_is_j)^{m_{ij}} \text{ for all } s_i, s_j\in S\rangle$$
for some finite set $S$ and where $M=(m_{ij})_{ij}$ is a symmetric matrix with $m_{ij}\in\{2,3,\dots \infty\}$ and with 1 in the diagonal entries. If the non-diagonal entries of $M$ are 2 or $\infty$, the group is called a \textit{right-angled Coxeter group}. A subset $J$ of $S$ is called $\textit{spherical}$ if the subgroup $W_J=\langle J\rangle \leq W$ is finite. In this case, the subgroup $W_J$ is called spherical. 

In  \cite[Lemma 5.4 and Theorem 5.5]{DegrijseLeary} is proved that, for all right-angled Coxeter group $W$, $H^n_W(X;K^0_W)=0$ for all $n>0$ and $H^0_W(X;K^0_W)=\Z^d$, where $X$ is a model for $\E W$ and $d$ is the number of spherical subgroups of $W$. Moreover, $K^1_W(\E W)=0$ and $K^0_W(\E W)\cong \Z^d$. We first extend this computation to equivariant $KO$-theory:

\begin{theorem}\label{cox1} Let $W$ be a right-angled Coxeter group with $d$ spherical subgroups. Then:
$$KO^{-n}_W(\E W)=\left\{\begin{array}{ll} \Z^{d}, & \text{ for } n\equiv 0,4 \text{ (mod 8),}\\
 (\Z_2)^{d}, & \text{ for } n\equiv 1,2 \text{ (mod 8),}\\
 0, & \text{ for } n\equiv 3,5,6,7 \text{ (mod 8).}
 \end{array}\right.$$
 \end{theorem}

\begin{proof}
Recall that any finite subgroup of an right-angled Coxeter group is isomorphic to $(\Z_2)^k$ for some $k$. Using the formula $R(G_1\times G_2)=R(G_1)\otimes R(G_2)$ (see section 3.2 of \cite{LinearRepresentation}), we get $R((\Z_2)^k)\cong \Z^{2^k}\cong RO((\Z_2)^k)$. The crucial fact that simplifies the computation is that the endomorphism field of any irreducible real representation of $(\Z_2)^k$ is $\R$, so using  formula (\ref{ecu2}) for $G=(\Z_2)^k$ we obtain that
$$KO^{-n}_G(\pt)\cong \Z^{2^k}$$
for $n\equiv 0,4$ (mod 8) and
$$KO^{-n}_G(\pt)\cong (\Z_2)^{2^k}$$
for $n\equiv 1,2$ (mod 8) and
$$KO^{-n}_G(\pt)=0$$
for $n\equiv 3,5,6,7$ (mod 8).

For $n\equiv 0,4$ (mod 8), the Bredon cochain complex of a  model $X$ for $\E W$ with coefficients in $KO^{-n}_W$ is equivalent to the one obtained with coefficients in $K^{0}_W$. From the previously known result recalled above, we get that $H^0_W(X;KO^{-n}_W)\cong\Z^d$ and zero in other degrees.
On the other hand, for $n\equiv 1,2$ (mod 8), if the Bredon cochain complex with coefficient in  $K^{0}_W$ is:
$$0\rightarrow \Z^{m_0}\xrightarrow{\delta_0}\Z^{m_1} \xrightarrow{\delta_1}\Z^{m_2}\rightarrow\cdots$$
then, tensoring by $\Z_2$, we get the Bredon cochain complex with coefficient in $KO^{-n}_W$:
$$0\rightarrow \Z^{m_0}\otimes \Z_2 \xrightarrow{\delta_0\otimes \Z_2}\Z^{m_1}\otimes \Z_2 \xrightarrow{\delta_1\otimes \Z_2}\Z^{m_2}\otimes \Z_2\rightarrow\cdots$$

The computation of the cohomology of this cochain can be easily done using the Universal Coefficient Theorem. 
Consider the chain complex $C_{-n}=\Hom_{\catname{Or}_\F(W)}({C_n}(X); KO^0_W)$ for $n\geq 0$ and $C_n=0$ for $n>0$.
If we tensor by $\Z_2$, the homology of $C_\bullet\otimes \Z_2$ is the same that the Bredon cohomology with coefficients in $KO^{-1}_W$ or $KO^{-2}_W$. The Universal Coefficient Theorem relates these two homology groups:

$$0\to H_n(C_\bullet)\otimes \Z_2\to H_n(C_\bullet\otimes \Z_2)\to \text{Tor}(H_{n-1}(C_\bullet),\Z_2)\to 0$$
is an exact sequence. Since all the Bredon cohomology groups computed above are torsion-free, the last term in the sequence is zero and we deduce that:
$$H_W^{n}(X; KO_W^{-1})\cong H_W^n(X;KO_W^{-2})\cong H_W^n(X;K^0_W)\otimes \Z_2$$

Note that, in any case, there is only one non-trivial column in the $E_2$-page of the Atiyah-Hirzebruch spectral sequence which trivially collapses.

Finally, the case $n\equiv 3,5,6,7$ (mod 8) is trivial, and hence, the proof is complete.
\end{proof}

\subsection{Non-right-angled Coxeter Groups}

Dealing with these groups is considerably more difficult than in the right-angled case. Nevertheless, we compute the $K$- and $KO$-theory of $\underline EW$ for certain non right-angled groups.  For a fixed $n$, let $W$ be the Coxeter group associated to the Coxeter matrix $M$ of size $n+1$  given by:

$$M=\qty(
\begin{array}{ccccccc}  2 & 3 & \infty&  & &&     \\
                                 3 & 2 & 3     &  \infty &  &\infty& \\
                                 \infty & 3& 2& 3 & \infty&  & \\
                                 & \ddots & \ddots &\ddots & \ddots & \ddots \\
                                 & & \infty &3 & 2 & 3 &\infty \\
                                 & \infty& & \infty & 3 & 2 &3\\
                                 & & & & \infty & 3 &2\\
\end{array})$$

Equivalently, this group can be given by its Coxeter-Dynkin diagram, where consecutive nodes are connected by an edge with a label 3 and non-consecutive nodes are connected by an edge with label $\infty$.

\begin{figure}[h!]
\centering
\begin{tikzpicture}
\draw[thick, fill=black] (0,0) circle [radius=0.1];
\draw[thick, fill=black] (2,0) circle [radius=0.1];
\draw[thick, fill=black] (4,0) circle [radius=0.1];

\node  at (5,0) {$\dots$};

\draw[thick, fill=black] (6,0) circle [radius=0.1];
\draw (0,0)--(4,0);

\draw[dashed] (0,0) to [out=80,in=100] (4,0);

\node at (2.1,1.4) {$\infty$};
\node at (2.5,2) {$\infty$};
\node at (4,-1.4) {$\infty$};

\node[below] at (0,-0.2) {$s_0$};
\node[below] at (1.8,-0.2) {$s_1$};
\node[below] at (4.2,-0.2) {$s_2$};
\node[below] at (6.2,-0.1) {$s_n$};

\node[above] at (1,0) {$3$};
\node[above] at (3,0) {$3$};

\draw[dashed] (0,0) to [out=85, in=95] (6,0);
\draw[dashed] (2,0) to [out=-80,in=-100] (6,0);

\end{tikzpicture}
\end{figure}

\begin{theorem}\label{non1}
$$\textstyle K^n_W(\E W)=\left\{\begin{array}{cl} \Z^{n+2}, & \text{$n$ even,}\\ 0, & \text{$n$ odd,} \end{array}\right.\qquad KO^{-n}_W(\E W)=\left\{\begin{array}{cl} \Z^{n+2}, & \text{$n \equiv 0,4$(mod $8$),}\\ (\Z_2)^{n+2}, &\text{$n \equiv 1,2$(mod $8$),}\\
		0, & \text{$n \equiv 3,5,6,7$(mod $8$).}
		 \end{array}\right.$$
\end{theorem}

For the proof we use  the {\em Bestvina complex} in \cite{Bestvina} which provides a particular model for $\underline EW$ for any general Coxeter group.

\begin{definition} A \textit{simple complex of finite groups} $G(\mathcal{Q})$ over a finite poset $\mathcal{Q}$ is a functor $P$, from the category associated to the poset $\mathcal{Q}$ to the category of finite groups, such that the non-identity morphisms are sent to injective non-surjective homomorphism.

For such simple complex of groups, the \textit{fundamental group} is the direct limit of this functor:

$$\widehat{G(\mathcal{Q})}
=\lim_{\substack{\longrightarrow\\ J\in\mathcal{Q}}} P(J)$$

\end{definition}

In the specific case of a Coxeter system $(W,S)$, we can associate the following simple complex of finite groups over the poset $\QQ$ of spherical subsets: for $J\subset S$, take $P(J)=\langle J\rangle\subset W$ and the inclusions of subsets $I\subset J\subset S$ induce inclusions of subgroups $P(I)=\langle I\rangle \hookrightarrow \langle J \rangle=P(J)$. It is clear that the fundamental group is equal to $W$.

\begin{definition}
A \textit{panel complex} over a finite poset $\QQ$ is a compact polyhedron $X$ with a family of subpolyhedrons $\{X_J\}_{J\in \QQ}$ called \textit{panels} such that $X$ is the union of the panels, $X_I\subset X_J$ if $J\leq I\in \QQ$ and the intersection of panels is empty or a union of panels.
\end{definition}

In \cite{Bestvina} two different kind of panel complexes are presented. Both are valid for the construction of a model for $\E W$, but the Bestvina complex gives a simpler model, in a geometric sense.

\begin{definition}
Given a finite poset $\QQ$, the \textit{Bestvina complex} $B$ is constructed by defining, for each maximal element $J$ of $\QQ$, $B_J$ to be a point. The following $B_J$ are constructed successively by defining $B_J$ as a compact contractible polyhedron containing $\bigcup_{I>J}B_I$ of the smallest possible dimension.
\end{definition}

\begin{definition}
Given a finite poset $\QQ$, a panel complex $(X,\{X_J\}_{J\in \QQ})$ over $\QQ$ and a simple complex of finite of groups $G(\QQ)$ over $\QQ$ with fundamental group $G$, the \textit{basic construction} is the space

$$D(X,G(\QQ))=(G\times X)/\sim$$

with the equivalence relation $(g_1,x)\sim(g_2,x)$ for all $x\in X$ and for $g_1,g_2\in G$ such that $g_{1}^{-1}g_2$ belongs to $P(J(x))$, where $X_{J(x)}$ is the intersection of all the panels containing $x$ (which is again a panel by the definition of panel complex).
\end{definition}

The following theorem establishes that this basic construction is a model for $\E W$. It can be obtained as a special case of the Theorem 5.2 of \cite{Bestvina}.

\begin{theorem}\label{Theorem Coxeter}
Let $(W,S)$ be a Coxeter system, with $\QQ$ the poset of spherical groups, $G(\QQ)$ the simple complex of finite groups defined above for such poset and $B$ the Bestvina complex over $\QQ$. Then $D(B,G(\QQ))$ is a model for $\E W$.\hfill$\square$
\end{theorem}
\begin{proof}[Proof of Theorem \ref{non1}]

Since in a Coxeter group the formula $\text{ord}(s_is_j)=m_{ij}$ holds
\cite[Section 1,  Chapter IV]{Bourbaki}, the only spherical subsets in $S$ are $\{s_i\}$ for $i=0,\ldots, n$  and $\{s_i,s_{i+1}\}$ for $i=0,\dots, n-1$. Hence, the spherical subgroups are $\langle \emptyset\rangle=\{ e \}$, $\langle s_i\rangle =\Z_2$ and $\langle s_i,s_{i+1}\rangle= A_2\cong \mathcal{S}_3$.
Thus, the Bestvina complex $B$ is a path of $n$ vertices:

 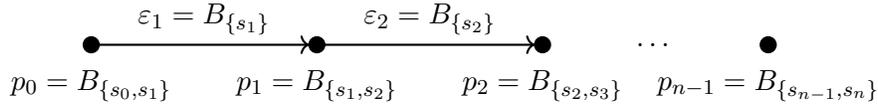
\begin{figure}[h!]
\centering
\begin{tikzpicture}

\draw[thick, fill=black] (0,0) circle [radius=0.1];
\draw[thick, fill=black] (3,0) circle [radius=0.1];
\draw[thick, fill=black] (6,0) circle [radius=0.1];

\node  at (7.5,0) {$\dots$};

\draw[thick, fill=black] (9,0) circle [radius=0.1];

\draw[ thick,->] (0,0)--(2.9,0);
\draw[ thick,->] (3,0)--(5.9,0);

\node[below] at (0,-0.2) {$p_0=B_{\{ s_0,s_1\} } $};
\node[below] at (3,-0.2) {$p_1=B_{\{ s_1,s_2\} } $};
\node[below] at (6,-0.2) {$p_2=B_{\{ s_2,s_3\} } $};
\node[below] at (9,-0.2) {$p_{n-1}=B_{\{ s_{n-1},s_n\} } $};

\node[above] at (1.5,0) {$\varepsilon_1=B_{\{s_1\}}$};
\node[above] at (4.5,0)    {$ \varepsilon_2=B_{\{s_2\}}$};

\end{tikzpicture}
\caption{Bestvina complex associated to the group $W$.}
\end{figure}
By Theorem \ref{Theorem Coxeter} the model $D(B,G(\QQ))=X$  for $\E W$ and its Bredon cohomology (with coefficient system $K^0_W$) is given by the following cochain complex:
$$0\rightarrow\bigoplus_{i=0}^{n-1} R(\mathcal{S}_3) \xrightarrow{\delta} \bigoplus_{i=1}^{n-1} R(\Z_2)\rightarrow 0$$
Using \cite{LinearRepresentation} as reference, it can be checked that $R(\mathcal{S}_3)=\Z^3$ and the inclusion  $\Z_2\cong\langle s_i\rangle \to \langle s_1,s_2\rangle\cong \mathcal{S}_3$, for $i=1,2$ is sent to $
f:\Z^3\cong R(\mathcal{S}_3)\to  \Z^2\cong R(\Z_2), (a,b,c)\mapsto (a+c,b+c)$. Then, 
$$\delta: (\Z^3)^{n}\to (\Z^2)^{n-1},\;\;\; (x_1,\ldots, x_n)\mapsto (f(x_2)-f(x_1), f(x_3)-f(x_2),\dots, f(x_n)-f(x_{n-1}))$$

Since the map $f$ is surjective, inductively, the map $\delta$ can be proved to be surjective, which implies that the only non-zero Bredon cohomology group is $H^0_W(\E W; K^0_W)=\Z^{n+2}$ which proves the first statement.

For the equivariant $KO$-theory observe that, since the endormorphism field of all the real irreducible representation of $\mathcal{S}_3$ is $\R$, the situation is analogous to the case of right-angled Coxeter groups: for $i=1$ or $i=2$, $H^n_W(\E W; KO^{-i}_W)\cong H^n_W(\E W; K^0_W)\otimes \Z_2$, and for $i=0$ or $i=4$, $H^n_W(\E W; KO^{-i}_W)\cong H^n_W(\E W; K^0_W)$. From this the result follows.
\end{proof}

In our last result we see that, if the Coxeter matrix is slightly perturbed, the $K$- and $KO$-theory drastically change. Consider $W$ the Coxeter group with the same matrix $M$ as before except that $m_{0,n}=3$ instead of $\infty$. Equivalently, in the Coxeter-Dynkin diagram, join the initial and the ending nodes by an edge with label 3. Then,

\begin{theorem}\label{non2}
$$K^n_W(\E W)=\left\{\begin{array}{cl} \Z^{n+3}, & \text{ $n$ even,}\\ \Z, & \text{$n$ odd,} \end{array}\right.\qquad KO^{-n}_W(\E W)=\left\{\begin{array}{cl} \Z_2\oplus \Z^{n+3}, &  \text{$n \equiv 0$(mod $8$),}\\
		A, &  \text{$n \equiv 1$(mod $8$),}\\ 
		(\Z_2)^{n+3}, & \text{$n \equiv 2$(mod $8$),}\\
		\Z_2, &  \text{$n \equiv 3$(mod $8$),}\\
		 \Z^{n+3}, & \text{$n \equiv 4$(mod $8$),}\\
		 0, &  \text{$n \equiv 5,6$(mod $8$),}\\
		 \Z_2, &  \text{$n \equiv 7$(mod $8$),}\\
		 \end{array}\right.$$
where $A$ is an extension of $\Z_2$ by $(\Z_2)^{n+3}$.
\end{theorem}

\begin{proof}
In this case a new spherical subset, $\{ s_0,s_n\}$, has appeared and the resulting  Bestvina complex is a polygon of $n+1$ sides.

\begin{figure}[h!]
\centering
\begin{tikzpicture}
\

\draw[thick] (-2,0)--(2,0);
\draw[thick] ({2*cos(60)},{2*sin(60)})--({-2*cos(60)},{2*sin(60)});
\draw[thick] ({2*cos(60)},{2*sin(60)})--({-2*cos(60)},{-2*sin(60)});
\draw[thick] ({2*cos(60)},{-2*sin(60)})--({-2*cos(60)},{-2*sin(60)});
\draw[thick] ({2*cos(60)},{-2*sin(60)})--({-2*cos(60)},{2*sin(60)});
\draw[thick] ({-2*cos(60)},{2*sin(60)})--(-2,0);
\draw[thick] ({2*cos(60)},{2*sin(60)})--(2,0);
\draw[thick] ({-2*cos(60)},{-2*sin(60)})--(-2,0);
\draw[thick] ({2*cos(60)},{-2*sin(60)})--(2,0);

\node[above left] at ({-2*cos(60)},{2*sin(60)}) {$B_{\{ s_0,s_1\}}$};
\node[above right] at ({2*cos(60)},{2*sin(60)}) {$B_{\{ s_1,s_2\}}$};
\node[below left] at ({-2*cos(60)},{-2*sin(60)}) {$B_{\{ s_4,s_5\}}$};
\node[below right] at ({2*cos(60)},{-2*sin(60)}) {$B_{\{ s_3,s_4\}}$};
\node[right] at (2,0) {$B_{\{ s_2,s_3\}}$};
\node[left] at (-2,0) {$B_{\{ s_5,s_0\}}$};

\node[left] at ({-2*cos(30)},{2*sin(30)}) {$B_{\{s_0\}}$};
\node[right] at ({2*cos(30)},{2*sin(30)}) {$B_{\{s_2\}}$};
\node[left] at ({-2*cos(30)},{-2*sin(30)}) {$B_{\{s_5\}}$};
\node[right] at ({2*cos(30)},{-2*sin(30)}) {$B_{\{s_3\}}$};
\node[above] at (0,2) {$B_{\{s_1\}}$};
\node[below] at (0,-2) {$B_{\{s_4\}}$};

\draw[thick, fill=lightgray] (-2,0)--({-2*cos(60)},{2*sin(60)})--({2*cos(60)},{2*sin(60)})--(2,0)--({2*cos(60)},{-2*sin(60)})--({-2*cos(60)},{-2*sin(60)})--(-2,0);
\draw[thick,<-, ,shorten <=4pt] (-2,0)--({-2*cos(60)},{2*sin(60)});
\draw[thick,<-, ,shorten <=4pt] ({-2*cos(60)},{2*sin(60)})--({2*cos(60)},{2*sin(60)});
\draw[thick,<-, ,shorten <=4pt] ({2*cos(60)},{2*sin(60)})--(2,0);
\draw[thick,<-, ,shorten <=4pt] (2,0)--({2*cos(60)},{-2*sin(60)});
\draw[thick,<-, ,shorten <=4pt] ({2*cos(60)},{-2*sin(60)})--({-2*cos(60)},{-2*sin(60)});
\draw[thick,<-, ,shorten <=4pt] ({-2*cos(60)},{-2*sin(60)})--(-2,0);

draw[thick, fill=black] (0,0) circle [radius=0.1];
\draw[thick, fill=black] (2,0) circle [radius=0.1];
\draw[thick, fill=black] (-2,0) circle [radius=0.1];
\draw[thick, fill=black] ({2*cos(60)},{2*sin(60)}) circle [radius=0.1];
\draw[thick, fill=black] ({2*cos(60)},{-2*sin(60)}) circle [radius=0.1];
\draw[thick, fill=black] ({-2*cos(60)},{-2*sin(60)}) circle [radius=0.1];
\draw[thick, fill=black] ({-2*cos(60)},{2*sin(60)}) circle [radius=0.1];

\node at (0,0) {$B_\emptyset$};

\end{tikzpicture}
\caption{Bestvina complex $B$ for the non-right-angled Coxeter group $W$ for $n=5$.}\label{Figure Hexagon 2}
\end{figure}

Analogously to the previous example, the basic construction gives a model  for $\E W$, and its Bredon cohomology is given by:
$$(\Z^3)^{n+1}\xrightarrow{\delta_0} (\Z^2)^{n+1}\xrightarrow{\delta_1} \Z$$
where $\delta_1$ is clearly surjective and $\delta_0$ is obtained from the transpose of the boundary map $\partial:\Z^{n+1}\to \Z^{n+1}$, given by the following matrix
$$\qty(
\begin{array}{rrrrrrrr}
1 & -1& 0          && \\
0 & 1 & -1 & 0 &  \\
   & 0 & 1  & -1        &0&\\
   &   &\ddots & \ddots & \ddots &\ddots &\\
   &&& 0 &1& -1& 0 \\
0 &&&&0 &1&-1\\
-1&0&&&&0&1\\
\end{array}
)
$$
and substituting the numbers $\pm 1$ by the function $\pm f$ defined above. The image of $\partial$ is the subgroup $H=\{(x_0,\dots, x_n)\in \Z^{n+1}\mid x_0+x_1+\dots+x_n=0\}\subset \Z^{n+1}$ and, since the map $f$ is surjective, the image of $\delta_0$ is isomorphic to $\Z^{2n}$. Therefore, the non-trivial Bredon cohomology groups are
$$H^0_W(\E W; K_W^0)=\Z^{n+3},\;\; H^1_W(\E W;K_W^0)=\Z.$$

Now, the spectral sequence have two non-trivial columns but all differential have to be trivial and the first statement follows.

Finally, for the equivariant $KO$-theory, recall that the endomorphism field is $\R$ for  all the real irreducible representation of these spherical subgroups. Hence, the Bredon cohomology with coefficient system $KO^{-i}_W$ can be easily computed used the information above and the Universal Coefficient Theorem:

\begin{table}[ht]
\centering
\begin{tabular}{c |  c  c  c  c  c  c  c  c}
$i$ & 0 &1 &2&$\;3\;$&$\;4\;$&$\;5\;$&$\;6\;$&$\;7\;$\\
[1ex]
\hline
\\
$H^{0}_{W}(\E W; KO_W^{-i})$  & $ \Z^{n+3}$ & $(\Z_2)^{n+3}$ & $(\Z_2)^{n+3}$ & 0 & $\Z^{n+3}$& 0&0&0\\
$H^1_W(\E W; KO_W^{-i})$ & $\Z$ & $\Z_2$ &  $\Z_2$ & 0 & $\Z$ &0&0&0
\end{tabular}
\end{table}

The $E_2$-page of the Atiyah-Hirzebruch spectral sequence is represented in the figure \ref{Figure spectral sequence x}. Note that there are not non-trivial differentials, so the spectral sequence trivially collapses and there is an extension problem in degree -1.

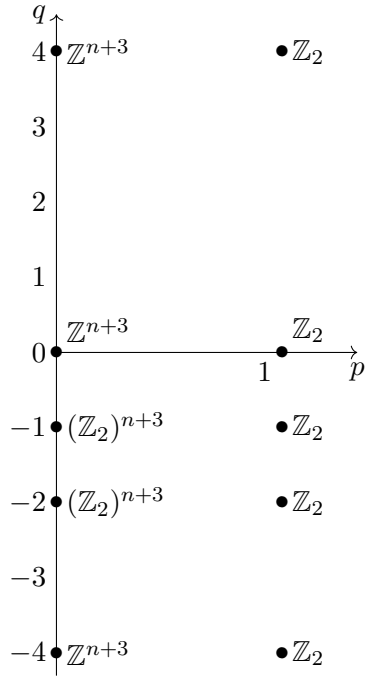
\begin{figure}[h!]
\centering
\begin{tikzpicture}
\draw[->] (0,-4.3)--(0,4.5);
\draw[->] (0,0)--(4,0);

\node[left] at (0,4) {$4$};
\node[left] at (0,3) {$3$};
\node[left] at (0,2) {$2$};
\node[left] at (0,1) {$1$};
\node[left] at (0,0) {$0$};
\node[left] at (0,-1) {$-1$};
\node[left] at (0,-2) {$-2$};
\node[left] at (0,-3) {$-3$};
\node[left] at (0,-4) {$-4$};

\node[right] at (0,4) {$\Z^{n+3}$};
\node[above right] at (0,0) {$\Z^{n+3}$};
\node[right] at (0,-1) {$(\Z_2)^{n+3}$};
\node[right] at (0,-2) {$(\Z_2)^{n+3}$};
\node[right] at (0,-4) {$\Z^{n+3}$};

\node at (0,4) {$\bullet$};
\node at (0,0) {$\bullet$};
\node at (0,-1) {$\bullet$};
\node at (0,-2) {$\bullet$};
\node at (0,-4) {$\bullet$};

\node[right] at (3,4) {$\Z_2$};
\node[above right] at (3,0) {$\Z_2$};
\node[right] at (3,-1) {$\Z_2$};
\node[right] at (3,-2) {$\Z_2$};
\node[right] at (3,-4) {$\Z_2$};

\node at (3,4) {$\bullet$};
\node at (3,0) {$\bullet$};
\node at (3,-1) {$\bullet$};
\node at (3,-2) {$\bullet$};
\node at (3,-4) {$\bullet$};

\node[below left] at (3,0) {$1$};

\node[left] at (0,4.5) {$q$};
\node[below] at (4,0) {$p$};

\end{tikzpicture}
\caption{First page of the spectral sequence, $E_2^{p,q}=H_W^p(\E W;KO_W^q)$.}\label{Figure spectral sequence x}
\end{figure}
\end{proof}

As explained in the introduction, this method can be applied to any Coxeter Group. For example, a similar computation for the infinite Coxeter Group $W=\tilde{A}_3$, gives as result $K^{\text{even}}_W(\E W)=\Z^9$ and $K^{\text{odd}}_W(\E W)=\Z^2$; for the group $W=\tilde{D}_4$, we get $K^{\text{even}}_W(\E W)=\Z^{33}$ and $K^{\text{odd}}_W(\E W)=0$. 
However, the models of such groups are more complicated ( of dimension 3 and 4 respectively), the representation theory involved is more difficult and the use of computers is necessary.

In particular, for these computations, the following software has been used: the system for computational discrete algebra, GAP (4.9.1), gives a list of irreducible representations of a finite group, with the commands \texttt{CharacterDegrees} and \texttt{IrreducibleRepresentations}. In order to obtain the finite subgroups of the Coxeter Groups previously mentioned, the CHEVIE project \cite{CH} has been used as a package of GAP: this software allows to get the character table of any finite Coxeter group, which is the first step in order to obtain the irreducible representations. Finally, using the Smith Normal Form with the software \texttt{MATLAB}, the Bredon cohomology can be computed from the matrices obtained in the previous step. 
Therefore, from a model of $\E W$ we can get the Bredon cohomology.

 When more complicated groups are considered (except in the case of having enough exponents $m_{ij}=\infty$ as in the examples of the previous section), the complexity of the models increases substantially, making the algorithm inefficient. Furthermore, in the examples of the previous section, extension problems associated to the spectral sequence have rarely appeared; when the dimension of the models grows, there are more non-trivial columns in the spectral sequence, and the extensions problems become frequent.



\textsc{Department of Algebra, Geometry and Topology, University of Malaga, Malaga 29080, Spain}

\textit{ Email address}: \texttt{mario.fuentes@uma.es}


\begin{thebibliography}{ZZ}

\bibitem{Apa}
M. P. G. Aparicio, P. Julg and A. Valette, {\itshape The Baum-Connes conjecture: and extended survey.}  In: Chamseddine
A., Consani C., Higson N., Khalkhali M., Moscovici H., Yu G. (eds) Advances in Noncommutative Geometry, (2019) Springer International Publishing.


\bibitem{AH}
M. F. Atiyah and  F. Hirzebruch, {\itshape Vector bundles and homogeneous spaces},  Proc. Sympos. Pure Math. Amer. Math. Soc. {\bf III}  (1961), 7--38.

\bibitem{Bar}
Barcenas, N. (2023). {\itshape A survey on Computations of Bredon Cohomology}. doi:10.48550/ARXIV.2302.10415

\bibitem{BCH94}
P. Baum, A. Connes and N. Higson, {\itshape Classifying space for proper actions and K-thoery of group}, in $C^*$-algebras: 1943-1993, Contemp. Math. {\bf 167} (1994), 241--292.

\bibitem{Bourbaki}
N. Bourbaki, {\itshape Lie Groups and Lie Algebras}, Springer  2002.

\bibitem{CoxeterRef2}
M. Bozejko   , T. Januszkiewicz and R. J. Spatzier , {\itshape Infinite Coxeter Groups do not have Kazhdan's property}, Journal of Operator Theory
Vol. 19, No. 1 (Winter 1988), pp. 63-67.

\bibitem{Coxetergroups}
M. W. Davis, {\itshape The geometry and topology of Coxeter groups}, Princeton University Press {\bf  32} (2008).

\bibitem{Assembly}
J. F. Davis and W. Lück,   {\itshape Spaces over a Category, Assembly Maps in Isomorphism Conjecture in $K$-and $L$-Theory}, K-Theory {\bf 15} (1998),  201--251.

\bibitem{qqq}
J. F. Davis and W. Lück, {\itshape
The topological K-theory of certain crystallographic groups},
Journal of Non-Commutative Geometry {\bf 7} (2013),
373--431.

\bibitem{DegrijseLeary}
D. Degrijse and I. J.  Leary, {\itshape Equivariant vector bundles over classifying spaces for proper actions}, Algebraic \& Geometric Topology {\bf 17}(1) (2017), 131--156.

\bibitem{CoxeterRef1}
N. Higson and G. Kasparov, {\itshape Operator $K$-theory for groups which act properly and isometrically on Hilbert space}, Electron. Res. Announc. Amer. Math. Soc. 3 (1997), 131-142.

\bibitem{Survey}
W.  Lück, {\itshape Survey on Classifying Spaces for Families of Subgroups}, in "Infinite Groups: Geometric, Combinatorial and Dynamical Aspects", Progress in Mathematics {\bf 248}, Birkhäuser (2005), 269--322.

\bibitem{WL}
W. Lück and R. Oliver, {\itshape The completion theorem in K-theory for proper actions of a discrete group}, Topology {\bf 40} (2001),  585--616.


\bibitem{Bestvina}
N. Petrosyan and T. Prytula, {\itshape Bestvina Complex for Group Actions with a Strict Fundamental Domain}, accepted in Groups, Geometry, and Dynamics, arXiv:1712.07606v2 (2018).



\bibitem{Segal}
G. Segal, {\itshape Equivariant K-theory}, Publications mathématiques de l'I.H.É.S. {\bf 34} (1968), 129--151.


	

\bibitem{LinearRepresentation}
J. P. Serre, {\itshape Linear Representations of Finite Groups}, Springer  1977.

\bibitem{Serre}
J. P. Serre, {\itshape Trees}, Springer 1980.






\bibitem[GHL{\etalchar{+}}96]{CH}
M.~Geck, G.~Hiss, F.~L{\"u}beck, G.~Malle, and G.~Pfeiffer.
\newblock {\sf CHEVIE} -- {A} system for computing and processing generic
  character tables for finite groups of {L}ie type, {W}eyl groups and {H}ecke
  algebras.
\newblock {\em Appl. Algebra Engrg. Comm. Comput.}, 7:175--210, 1996.












\end{thebibliography}
\end{document}